\documentclass[12pt, draft, regno]{amsart} 
\usepackage{latexsym} 
\usepackage{pst-all}
\usepackage[Conny]{fncychap}
\usepackage{eufrak}
\usepackage[psamsfonts]{amssymb} 
\usepackage{amsmath}
\usepackage{enumerate}
\usepackage{amsfonts}
\usepackage{hyperref}  
\usepackage{mathabx}

\usepackage{lipsum}

\newcommand\blfootnote[1]{%
  \begingroup
  \renewcommand\thefootnote{}\footnote{#1}%
  \addtocounter{footnote}{-1}%
  \endgroup
}

\usepackage[latin1]{inputenc}  
\usepackage[T1]{fontenc}        
\usepackage{url}               
\usepackage{amsmath}
\usepackage{lmodern} 

 \usepackage{pst-all}

\newtheorem{theorem}{Theorem}[section]
\newtheorem{lemma}[theorem]{Lemma}
\newtheorem{proposition}[theorem]{Proposition}
\newtheorem{corollary}[theorem]{Corollary}
\theoremstyle{definition}
\newtheorem{definition}[theorem]{Definition}
\newtheorem{example}[theorem]{Example}
\theoremstyle{remark}

\usepackage{enumerate}
\usepackage[psamsfonts]{amssymb} 
\usepackage{enumerate}
\numberwithin{equation}{section}
\begin{document}

\title[]{Contraction criteria for Brownian filtrations samplings}

\author{R\'{e}mi  Lassalle}
\address{Universit\'{e} Paris-Dauphine, PSL, Place du Mar\'echal De Lattre De Tassigny, 75775 Paris Cedex 16, France}
\email{lassalle@ceremade.dauphine.fr}

\vspace{0.5cm}

\begin{abstract}
This paper investigates the problem to determine whether a given stochastic process generates a sampled Brownian filtration. A fairly general sufficient condition is obtained by applying the Frank H. Clarke contraction criteria to  a functional whose construction relies on the quasi-invariance properties of a sampled Wiener measure. The latter are investigated from the precise analytic structure underlying this sampled Brownian motion. In particular, the Cameron-Martin space is identified from usual Lebesgue integrals over a Guseinov measure. As an application we obtain sufficient conditions for the  existence of a solution to a class of not necessarily Markovian stochastic differential equations driven by a sampled Brownian motion. By taking a  sampling times set which coincides with the unit interval, the results apply to continuous time models.
\end{abstract}

\maketitle

\noindent

\noindent 

\section{Introduction}
Nonlinear and functional analysis adapted for measured spaces endowed with an increasing collection of $\sigma-$ fields is certainly an essential thread of stochastic analysis. 
In this paper, we investigate fixed point methods applied to $\sigma-$ fields coincidence problems, within the scope of stochastic analysis with time scales samplings.  In particular, we show how to apply the sharp Clarke fixed point theorem (\cite{CLARKE1}) of nonlinear analysis to obtain sufficient conditions for a given continuous stochastic process $(X_t)_{t\in \mathbb{T}}$,  which stands for the \textit{data} of our problem,  to be a generating sampled Brownian motion  on a completed naturally filtered canonical path space, which is endowed  with a Borel probability measure $\nu$,  which stands for the \textit{unknown} of the problem; here, as it is usual, a filtration means an increasing collection of $\sigma-$ fields, while the time scales $\mathbb{T}$ of interest here include the unit interval $\mathbb{T}=[0,1]$ as a particular case. One of the motivations to find such probabilities, in this latter case, is that the process $(X_t)_{t\in[0,1]}$ has not only the predictable representation property, but also the chaotic representation property (\cite{MMYOR}) on the canonical space endowed with $\nu$.

 To briefly motivate the framework, recall that within the perspective of stochastic calculus, continuous time models as those indexed by $[0,1]$ are well known to yield a technical gap with respect to usual discrete time models indexed by subsets of non-negative integers. This gap is filled by the introduction of considerably much technical tools as predictable processes and of much refined conditioning strategies (see \cite{DM}, \cite{D3}, \cite{I-W}, \cite{JACOD}). Those features are particularly emphasized by the problematic of existence of strong solutions to continuous time stochastic differential equations, and from the so-called first Tsirelson counterexample (\cite{LEGALLYOR}, \cite{Tsi}). Starting from a strictly decreasing sequence \begin{equation} \label{123L1} \{ t_n : n\in \mathbb{Z}_{-} \}\subset [0,1],\end{equation} such that  \begin{equation} \label{123L2} \lim_{n\to -\infty} t_n =0, \end{equation} the latter provides an example of a Borel probability measure $\nu_{T}$ $\in$ $M_1$ $(C_{[0,1]})$, equivalent to the classical Wiener measure $\mu_W$ (\cite{WIENER1}),  with the predictable representation property with respect a Brownian motion $(B_t^{\nu_{T}})$ of the augmented canonical filtration $(\mathcal{F}_t^{\nu_{T}})$, which is however not generated by this Brownian motion $(B_t^{\nu_{T}})$ on $(C_{[0,1]}, \mathcal{B}(C_{[0,1]})^{\nu_T}, \nu_T)$ (see \cite{YOR2}); the augmented filtration $(\mathcal{G}_t^{B^{\nu_{T}}})$ generated by the Brownian motion $(B_t^{\nu_{T}})$ is strictly smaller than  $(\mathcal{F}_t^{\nu_{T}})$ which we denote by $$(\mathcal{G}_t^{B^{\nu_{T}}}) \subsetneq (\mathcal{F}_t^{\nu_{T}}).$$ Within the same line, the Dubins-Feldman-Smorodinsky-Tsirelson (\cite{DUBINS}) counter $-$ example, which answers by the negative to a question of \cite{STROOCK3},  provides a Borel probability measure $\nu_{D}\in M_1(C_{[0,1]})$ equivalent to the Wiener measure, which still have the predictable representation property with respect to a $(\mathcal{F}_t^{\nu_{D}})-$ canonical Brownian motion $(B_t^{\nu_{D}})$, and further enjoys stronger specificities as $\nu_T$. In this latter case, which is based on the Vershik theory (\cite{VERSHIK}), not only $(B_t^{\nu_{D}})$ does not generate the augmented canonical filtration, but also the inclusion   $$(\mathcal{G}_t^{B}) \subsetneq (\mathcal{F}_t^{\nu_D})$$
is strict for all $(\mathcal{F}_t^{\nu_{D}})$ Brownian motion $(B_t)$ on $(C_{[0,1]}, \mathcal{B}(C_{[0,1]})^{\nu_D}, \nu_D).$

 This counterexample has been the subject of several enlightening (for instance see \cite{EMERY6}, \cite{SCHACHERMAYER}), and as it is pointed out in \cite{DUBINS} (see also \cite{YOR2}), in the main lines the argument turns out to the existence of a strictly decreasing sequence satisfying~(\ref{123L1}) and~(\ref{123L2}) which is such that $(\mathcal{F}_{t_k}^{\nu_D})_{k\in \mathbb{Z}_{-}}$ is not generated by a family of independent random variables  $(X_k)_{k\in \mathbb{Z}_{-}}$ with diffuse laws. At the inverse,  those sharp negative results and counterexamples strongly motivate for the search of positive results to the problem investigated in the present paper, within a framework which somehow interpolates between the discrete and the continuous time models.  Seeking for a framework which provides a point of view allowing as well finite time models, strictly decreasing sequences of the form~(\ref{123L1}) and~(\ref{123L2}), and the continuous framework $[0,1]$, while keeping sufficient topological properties to perform stochastic analysis, we will work with time scales $\mathbb{T}$. Recall that those objects appear as key tools in the calculus of measure chains introduced by Stefan Hilger (\cite{HILGER}) to provide a unifying point of view on discrete and continuous time models of analysis (see \cite{HILGER}, see also \cite{Bohner}). Ever since, the tool box of time scales analysis  has considerably increased, and we recall that its very basic object  is a closed subset $\mathbb{T}$ of $\mathbb{R}$ packed together with a specific classification of the elements of $\mathbb{T}$,  which will be recalled below.  Among several works involved in stochastic calculus with time scales, we particularly mention \cite{BOHNER3}, \cite{GROW}, \cite{LUNGAN},  \cite{SANYAL}, and we particularly emphasize \cite{BOHNER2}. The latter, which considers $\mathbb{T}$ sampled Wiener processes, notably proposes to construct stochastic integrals from specific extended functions, and further consider solutions to stochastic differential equations driven by a sampled Brownian motion with Markovian drift terms, under suitable Lipschitz conditions, within the specific construction of \cite{BOHNER2}. However, to obtain the required analytic structure of the sampled Brownian motion, which seems necessary for our specific purposes, we rather follow a usual route in stochastic analysis, which goes through abstract Wiener spaces theory, and starts from the Cameron-Martin space. 
 To construct the latter, we will take advantage of the measure $\lambda_{\Delta}$ of \cite{BOHNER2}, interpreted as a Borel probability measure of Guseinov's type \cite{GUSEINOV}, which seems to be much accessible to probabilists; recall that when $\mathbb{T}=[0,1]$, $\lambda_{\Delta}$ boils down to the standard Lebesgue measure $\lambda$.   At this stage, it seems necessary to point out that, for the sake of consistency with the classical continuous time stochastic analysis framework and for the sake of clarity, we will make specific assumptions and slightly adapt the notation, while keeping a consistency with \cite{BOHNER2}, whose integrals are shown here to provide the same values as suitable Paley-Wiener integrals on usual indicator functions.  For instance, within the whole paper, we will work with  a time scale $\mathbb{T}$ which satisfies $$\{0,1\}\subset \mathbb{T}\subset [0,1].$$ Under the fundamental time scales classification recalled below, whenever $\mathbb{T}$ contains a sequence which satisfies~(\ref{123L1}) and~(\ref{123L2}), $0$ is an element of the subset of right-dense points $\mathbb{T}_{rd}$ of $\mathbb{T}$.  Therefore, when handled as \textit{random functions}, our continuous stochastic processes will be random elements taking their values in a suitable space $C_{\mathbb{T}}$ of real valued continuous functions defined on $\mathbb{T}$. The latter is still endowed with the norm of uniform convergence thus yielding a Borel $\sigma-$field $\mathcal{B}(C_{\mathbb{T}})$, while the law of $X$, as a random function, will be a Borel probability measure on the possibly completed measurable space $(C_{\mathbb{T}}, \mathcal{B}(C_{\mathbb{T}}))$. Given a Borel probability measure $\nu\in M_1(C_{\mathbb{T}})$ on $(C_{\mathbb{T}}, \mathcal{B}(C_{\mathbb{T}}))$,  $(\mathcal{F}_t^\nu)$ denotes the augmentation of the filtration generated by the evaluation process $(W_t)_{t\in \mathbb{T}}$. More accurately, within this specific paper, as we are essentially interested in sampled brownian filtrations, we set $$\mathcal{F}_t^\nu:=\sigma(W_s : s\in [0,t]_{\mathbb{T}})^\nu, \forall t\in \mathbb{T},$$ which at each $t\in \mathbb{T}$ denotes the completion with respect to the whole $\nu$ negligible sets of the $\sigma-$ field generated by the $\{W_s,  s\in [0,t]_{\mathbb{T}} \}$, where $W_s : \omega\in C_{\mathbb{T}}\to \omega(s)\in \mathbb{R}$, denotes the evaluation function for any $s\in \mathbb{T}$, and where $[0,t]_{\mathbb{T}}:=[0,t]\cap \mathbb{T}$.

 Within the abstract form of a \textit{problem of measure}, problems as those stated in the first paragraph have been addressed notably in \cite{ERSHOVPHD}, \cite{ERSHOV2}, where general results of theory of measure, as those of \cite{MARCZEWSKI} and of \cite{SION},  have been showed to yield efficiently solutions in much general context to less specified versions of this problem, called Yershov problems : within our specific context and datas, this problem amounts to find $\nu\in M_1(C_{\mathbb{T}})$ such that $(X_t)_{t\in \mathbb{T}}$ is a sampled Brownian motion with respect to its own filtration. We refer to the well documented monography of V.I. Bogachev \cite{BOGACHEV} for further references, where several works of Yershov have been put into light much recently under an abstract form. In particular, our unknown probability measure  $\nu\in M_1(C_{\mathbb{T}})$ induces an \textit{extremal solution} to a corresponding Yershov problem, that is within our context,  this solution also satisfies $\mathcal{G}_1^X= \mathcal{F}_1^\nu;$ $\mathcal{G}_1^X= \sigma(X_s : s\in [0,1]_{\mathbb{T}})^\nu$. However, despite the efficiency of applications of general principles of measure theory in stochastic frameworks to obtain sufficient conditions for the existence of a probability $\nu$ which turns the process $(X_t)$ into a Brownian motion with respect to its own filtration, existing works following this line rather emphasize the importance of tools of functional analysis and of nonlinear analysis to obtain much precise results and sufficient conditions of existence to extremal solutions, and  to non-anticipating solutions which further satisfy the much technical coincidence of filtrations  \begin{equation} \label{filtrcoinc12} \mathcal{G}_t^X =\mathcal{F}_t^\nu, \ \forall t\in \mathbb{T}.\end{equation} In particular, in continuous time frameworks, several works of V.A. Benes as \cite{INNOVATION2}, \cite{INNOVATIONB2}, which investigate problems related to the Yershov problem in much analytic frameworks, emphasize the importance of tools of analysis to solve coincidence of $
 \sigma-$ fields problems, which furthermore allow to obtain extremal solutions to the Yershov problem; the latter also points out connections with the Zvonkin theorem  (\cite{ZVONKIN}) and applications to the innovation conjecture (\cite{FROST},\cite{INNOVATION3}), a rigorous exposition of which is provided by the nowadays classical paper of the french mathematician P.A. Meyer \cite{Meyer2}.  On the other hand,  we notably emphasize works of M. Yor \cite{YOR0},  where a theorem of functional analysis (\cite{DOUGLAS}) 
  plays the first roles to obtain a filtrations coincidence, and also works of D. F. Allinger  who proposed to investigate much  rigorously related problems of coincidence of filtrations as~(\ref{filtrcoinc12}) with nonlinear operators, in  view of applications to the innovation problem \cite{INNOVATION1}, \cite{INNOVATIONALLINGER2}, \cite{INNOVATION3ALI} ;  those latter aspects somehow appear as a shadow, while their main technical points and rigorous hypothesis remain hidden behind the curtain in several other works.
 Thanks to the development of tools 
 of \textit{stochastic analysis} (\cite{MALLIAVIN}) and particularly of \textit{Malliavin calculus} (see \cite{CRUZEIROMALLIAVIN09}, \cite{MALLIAVIN7}, \cite{MALLIAVIN}, \cite{NOURDIN}, \cite{USTUNELBOOKZ}), approaches based on functional analysis and on nonlinear analysis methods to solve distinct 
 problems which, as a byproduct,  can be used to obtain extremal solutions to Yershov problems  on $\mathbb{T}=[0,1]$, and some solutions which satisfy~(\ref{filtrcoinc12}), have 
 been considerably developed ever since. As far as fixed point methods are involved, we particularly mention a sequence of works of notably D.Feyel, A.S. \"{U}st\"{u}nel, and M. Zakai,  such as~(\cite{FUZ}) and (\cite{USTUZAKAIINV1}) which also review several related approaches; for other related problems see also \cite{JFALAS}, and the references therein.  However, contrary to those among those previous works whose proofs are built from a fixed point theorems, as the Leray-Schauder theorem (see \cite{FUZ}), our conditions will not be stated pathwise, but will be global integral conditions. The reason for this approach is to obtain rigorously the whole coincidence of filtrations~(\ref{filtrcoinc12}), without neither adding further conditions to ensure a specific uniqueness of solutions to the corresponding Yershov problem, nor a localization of the constraints (for instance, see \cite{INNOVATIONALLINGER2} with pathwise approaches, \cite{LOCINVSTO} with integral approaches).   Following this line, we obtain Theorem~\ref{theorem2}, which is our main result. Under precise hypothesis, the latter provides a sufficient condition based on the fixed point theorem of the analyst Frank H. Clarke (\cite{CLARKE1}), for $(X_t)_{t\in\mathbb{T}}$ to be a generating sampled Brownian motion with respect to some Borel probability measure $\nu\in M_1(C_{\mathbb{T}})$ whose existence is obtained;    the contraction criteria of \cite{CLARKE1}  is also one of the well known applications of Ekeland's variational principle (\cite{EKELANDVAR}) has shown in \cite{EKELAND2}, where a clear statement of this criteria is also recalled.  As an application, we provide sufficient conditions of existence for a strong solution to not necessarily Markovian stochastic differential equations driven by a sampled Brownian motion $(B_t)_{t\in \mathbb{T}}$,  of the form  \begin{equation}X_t := B_t + \int_{[0,t)_{\mathbb{T}}} \beta_s(X) \lambda_\Delta(ds), \ \forall t\in \mathbb{T},\label{edsINTRO1} \end{equation} stemming from an application of the Frank H. Clarke fixed point theorem. These solutions are obtained by a rigorous method adapted to our modern context, which is based on a pioneering vision of the french mathematician Robert Fortet in \cite{Fortet2}. Written in the first half of the XXth century, this paper which is well known to be fundamental in the study of absolutely continuous transport of laws of stochastic processes (\cite{BRUYOR}), proposed notably a \textit{symbolic reciprocal formula} to handle an equation with randomness from a fixed point method.

 The structure of the paper is the following.  Section~\ref{P} introduces the notation of the paper and provides a brief recall on the tools required subsequently which are essentially of two distinct kinds, from time scales and from stochastic analysis. As reading this paper does not require to have prior knowledge on time scales analysis,  we provide a recall on those specific tools necessary to follow the proofs in a way which is much suitable to probabilists. The  integrals of Hilger's time scales analysis, which seem necessary for our purposes, play a key r\^{o}le to obtain the analytic structure of the sampled Brownian motion from the Gross construction of abstract Wiener spaces (\cite{gr65}).  The reason to identify those features is that, within the specific approach which we unroll, the function whose fixed points will be of interests require sharp results on the absolute continuity of specific transports of measure to be well defined. Therefore, we require not only the sampled Brownian motion process, but also its whole analytic structure, including its Cameron-Martin space which is suitably identified by using integrals over  Guseinov's type measures  $\lambda_{\Delta}$ as recalled together with the Paley-Wiener integrals in Section~\ref{ANALYTICBROWNIAN}, whose values on usual specific indicators are easily checked to be consistent with those of the construction of \cite{BOHNER2}. Once the desired Cameron-Martin space of the sampled Brownian motion has been identified, we infer the transformations of interest, notably those which leave the sampled Wiener measure quasi-invariant.  This task is achieved in section~\ref{CML} where, after a recall of fundamental results on the quasi-invariance of the sampled Wiener measure, Lemma~\ref{Lemmabs}, which is a Kailath-Zakai type lemma on absolute continuity, is obtained. From this lemma, in Section~\ref{SCC} we state our main result Theorem~\ref{theorem2}. The latter, which is based on the Clarke fixed point theorem, provides a sufficient condition for a continuous stochastic process $(X_t)_{t\in \mathbb{T}}$ defined on $C_{\mathbb{T}}$ to be a generating sampled Brownian motion on the naturally filtered probability space $(C_{\mathbb{T}},\mathcal{B}(C_{\mathbb{T}})^\nu, \nu)$, for some Borel probability measure $\nu \in M_1(C_{\mathbb{T}})$.  Finally,   section \ref{Fortetsection} states  Corollary~\ref{sdecorollary}, which provides sufficient conditions for the existence of a strong solution to~(\ref{edsINTRO1}) from the Clarke fixed point theorem. Explicit examples of solutions to those equations, with a not necessarily neither Markovian nor delayed drift term, and which generate sampled Brownian filtrations are provided, notably  with the Cantor set $K_3$ which is a typical time scale.

\section{Preliminaries and notation}
\label{P}

\subsection{Recall and notation of time-scales $\mathbb{T}$}

Time scales theory has been introduced by Stefan Hilger (see \cite{HILGER}), we refer to \cite{Bohner} for a clear survey on part of this dynamic domain whose literature seems to grow increasingly.

In the whole paper $\mathbb{T}$ denotes a closed subset of $\mathbb{R}$, which is further assumed to satisfy $\{0,1\} \subset \mathbb{T} \subset [0,1]$; the latter hypothesis which is quite specific to our stochastic framework considerably clarify the notation, while keeping a viewpoint sufficiently general to obtain interesting applications. Such set $\mathbb{T}$,  packed together with the specific classification recalled in Definition~\ref{Hilgerclasse}, is called a time scale, and is one of the main tools and symbols of the time scales analysis, since the seminal works of Hilger (\cite{HILGER}).  The latter provides an unifying point of view on discrete and continuous time models of analysis.  Within this specific framework,  we first recall basic aspects and tools of time scales analysis,  sufficient for our purposes. However, due to our requirement to handle the whole analytic structure of the $\mathbb{T}-$ sampled Brownian motion, which is often hidden in probability, notably the associated Cameron-Martin space,  we will have to slightly adapt the notation.  In the whole paper,  we adopt the notation  $I_\mathbb{T}:= I\cap \mathbb{T}$ from \cite{BOHNERNOTATION}, when $I\subset [0,1]$ denotes an interval. 

We now recall how to obtain the measure $\lambda_{\Delta}$ of \cite{BOHNER2} within the specific circumstances encountered in this paper. The latter is interpreted as a Guseinov measure (\cite{GUSEINOV}), from the Rynne (\cite{Rynne}) approach, which is to obtain those latter measures as pushforwards of the standard Lebesgue measure by suitable measurable functions. First,  
the set $\mathbb{T}$ is endowed with its usual induced topology as a subset of $\mathbb{R}$, while $\mathcal{B}(\mathbb{T})$ denotes the associated Borel $\sigma-$field on $\mathbb{T}$, which coincides with the trace of $\mathcal{B}([0,1])$ on $\mathbb{T}$; in particular, $\mathcal{B}([0,1])$ denotes this $\sigma-$ field when $\mathbb{T}=[0,1]$. The Guseinov measure $\lambda_{\Delta}:=\bar{\rho}_\star \lambda$ of interest is a Borel probability measure on the measurable space $(\mathbb{T}, \mathcal{B}(\mathbb{T}))$, which is obtained as the direct image (pushforward) $\bar{\rho}_\star \lambda$ of the classical Lebesgue measure $\lambda$ on $([0,1],\mathcal{B}([0,1]))$, by the $\mathcal{B}([0,1]) /\mathcal{B}(\mathbb{T})$  Borel measurable function $\bar{\rho} : [0,1] \to \mathbb{T}$ which is given by $$\bar{\rho}(t):= \begin{cases}  \sup(\{ u < t : \ u\in \mathbb{T}\}) \  if \ t\in ]0,1] \\ 0 \   if  \ t = 0 \end{cases},$$  so that it satisfies $$\lambda_{\Delta}(A) := \bar{\rho}_\star \lambda(A) = \lambda(\bar{\rho}^{-1}(A)):=\lambda(\{s\in [0,1] : \bar{\rho}(s)\in A\}), \ \forall A\in \mathcal{B}(\mathbb{T}),$$ while $\bar{\rho}$ takes values in $\mathbb{T}$ since the latter is a closed set. Notice that we used a strict inequality in the definition of $\bar{\rho}$, and that due to the specificities of the framework, we have also denoted $\bar{\rho}(0)=0$. The restriction $\rho:=\bar{\rho}|_{\mathbb{T}}$ of $\bar{\rho}$ to $\mathbb{T}$ is the  so- called the \textit{backward jump operator} of time scales analysis.

Another function of interest is given by  $\bar{\sigma} : [0,1] \to \mathbb{T}$, the function defined by $$\bar{\sigma}(t) := \begin{cases}  \inf(\{ u> t : u\in 
\mathbb{T}\}) \  if \ t\in [0,1[  \\ 1 \   if  \ t = 1 \end{cases}.$$ The restriction $\bar{\sigma}|_{\mathbb{T}}$ of this function to $\mathbb{T}$ is the so-called \textit{forward jump operator}, which is used here to compute explicitly 
values of $\lambda_{\Delta}(A)$ for specific sets of the form $A:=I\cap \mathbb{T}$, $I$ denoting an interval of $[0,1]$; notice that for the sake of consistency with usual notations of 
stochastic analysis in the case where we take $\mathbb{T}:=[0,1]$, we have set $\bar{\sigma}(1)=1$. From the definitions together with the Hilger's time scales classification recalled below, the functions $\bar{\rho}$ and $\bar{\sigma}$ are easily checked to be related by the formula\begin{equation} \bar{\rho}^{-1}([0,t]_{\mathbb{T}}) = [0,\bar{\sigma}(t)], \ \forall t\in \mathbb{T}, \label{inversima} \end{equation}  which determines the values of $\lambda_{\Delta}$, where $\bar{\rho}^{-1}([0,t]_{\mathbb{T}}) :=\{s \in [0,1] : \bar{\rho}(s)\in [0,t]_{\mathbb{T}}\}$, denotes an inverse image for each $t\in \mathbb{T}$.  Aside the backward jump operator, and the forward jump operator, another usual tool in time scales analysis is the so-called \textit{graininess} of $\mathbb{T}$, which we denote here by $\bar{\mu}_g$ to distinguish it from the notation of the measures which we handle. Recall that, within our specific hypothesis on $\mathbb{T}$,  the graininess of $\mathbb{T}$ is the function $\bar{\mu}_g : \mathbb{T} \to \mathbb{R}$ defined by $$\bar{\mu}_g(t) := \bar{\sigma}(t)- t, \  \forall t\in \mathbb{T}.$$ The graininess and the function $\bar{\sigma}$ enable to identify those values of the Guseinov measure $\lambda_{\Delta}$ necessary to follow the proofs of this paper   \begin{equation} \label{guseinovvalues} \lambda_{\Delta}(I_{\mathbb{T}})=  \begin{cases} \bar{\mu}_{g}(t) \ if \ I_{\mathbb{T}}=\{t\}, \ \ t\in \mathbb{T} \\ \bar{\sigma}(t)- s \ if  \ I_{\mathbb{T}}= [s,t]_{\mathbb{T}},  \ s,t\in \mathbb{T}, \ s < t ,  \end{cases}\end{equation}  which follows from the definition together with~(\ref{inversima}); values of other scaled intervals $I_{\mathbb{T}}$, where $I$ denotes an interval of $[0,1]$, easily follow from~(\ref{guseinovvalues}) together with the $\sigma-$ additivity of $\lambda_{\Delta}$. Those operators are also of importance to make the usual distinctions of time scales, which we recall here within our specific hypothesis :
\begin{definition} (Hilger's time scales classification.) \label{Hilgerclasse}
Assuming that $\mathbb{T}$ is a time scale such that $\{0,1\}\subset \mathbb{T} \subset[0,1]$, a point $t\in \mathbb{T}$ is said to be
\begin{itemize}
\item right-scattered, if $\bar{\mu}_g(t)>0$, and $t\in \mathbb{T}\setminus \{1\}$
\item right-dense, if $\bar{\mu}_g(t)=0$ and $t\in \mathbb{T}\setminus \{1\}$
\item left-scattered, if $\rho(t)<t$ and  $t\in \mathbb{T}\setminus \{0\}$
\item left-dense if $\rho(t)=t$ and  $t\in \mathbb{T}\setminus \{0\}$
\end{itemize}
Subsequently, $\mathbb{T}_{rs}$ (resp. $\mathbb{T}_{ls}$) denotes the at-most countable subset of right-scattered (resp. left-scattered) points of $\mathbb{T}$.
\end{definition}
Notice that the subset $\mathbb{T}_{ls}\subset[0,1]$ of left-scattered points is a null set for the standard Lebesgue measure $\lambda$ on $([0,1],\mathcal{B}([0,1]))$, which has no atom, while from the definitions $\bar{\rho}(t)=t$ holds, $\forall t\in \mathbb{T}\setminus \mathbb{T}_{ls}$. This easily entails $$\int_{\mathbb{T}} f(t)\lambda_{\Delta}(dt) = \int_{[0,1]} f(\sup([0,t]_{\mathbb{T}})) \lambda(dt), \ \forall f\in \mathcal{L}^1(\lambda_{\Delta}),$$ which allows us to identify the measure of \cite{BOHNER2} within our specific hypothesis as it has been announced above, while using the function $\bar{\rho}$ to define $\lambda_{\Delta}$ is justified by the nice properties~(\ref{inversima}) of its inverse images of scaled intervals; following the usual notation $\mathcal{L}^1(\lambda_{\Delta})$ denotes the set of $\mathcal{B}(\mathbb{T})/\mathcal{B}(\mathbb{R})$ measurable functions $f :\mathbb{T}\to \mathbb{R}$ such that $\int_{\mathbb{T}}|f(t)|\lambda_{\Delta}(dt)<+\infty$, while as it was recalled, $\lambda_{\Delta}$ is a Borel measure on the measurable space $(\mathbb{T}, \mathcal{B}(\mathbb{T}))$.  A key  tool to handle $\lambda_{\Delta}$ below is the Lebesgue decomposition formula for the Guseinov measure $\lambda_{\Delta}$. The latter may be found for instance in the Proposition 2.1 of \cite{Agarwal}, which under our hypothesis and notation holds under the form  
\begin{equation} \lambda_{\Delta} = \lambda|_{\mathcal{B}(\mathbb{T})} + \sum_{t\in \mathbb{T}_{rs}} \bar{\mu}_g(t) \delta^{Dirac}_t, \label{Hilgerdecomp} \end{equation}
$\lambda|_{\mathcal{B}(\mathbb{T})}$ denoting the restriction of the usual Lebesgue measure on $([0,1],$ $\mathcal{B}([0,1]))$ to the sub sigma-field $\mathcal{B}(\mathbb{T})$, $\delta^{Dirac}_t$ denoting the Dirac measure concentrated on  an element $t$ of the at-most countable set $\mathbb{T}_{rs}$, $\bar{\mu}_g$ denoting the graininess of $\mathbb{T}$. To get a straightforward intuition of the origin of~(\ref{Hilgerdecomp}), from the $\sigma-$ additivity together with the definition of  $\lambda_{\Delta}$, it is enough to notice that $\mathbb{T}\cap \bar{\rho}^{-1}(\mathbb{T}_{rs})$ is at-most countable and therefore $\lambda$ negligible, that $\bar{\rho}^{-1}(\mathbb{T})\cap \mathbb{T} =\mathbb{T}$, while $\bar{\rho}(t) \in \mathbb{T}\setminus \mathbb{T}_{rs}$ implies $\bar{\rho}(t)=t$, for all $t\in [0,1]$, which can be checked readily by contradiction. 
Subsequently the previous decomposition formula~(\ref{Hilgerdecomp}) may be used to identify integrals over  a Guseinov measure. Given $f\in \mathcal{L}^1(\lambda_{\Delta})$ recall that the time scales integral of $f$ might be also denoted by $$\int_{\mathbb{T}} f(t) \Delta t := \int_{\mathbb{T}} f(t) \lambda_\Delta(dt), \ \forall f\in \mathcal{L}^1(\lambda_{\Delta}) .$$ From the previous decomposition formula, whenever $\mathbb{T}=[0,1]$ it boils down to an integral over the standard Lebesgue measure. We refer to \cite{GUSEINOV} and to the references therein for the relationship of such integrals with other much specific integrals of time scales analysis, as time scaled Riemann integrals, and with the Hilger derivative.

\subsection{The sampled Wiener measure $\mu_{\mathbb{T}}$}

The measures which we consider here will be mostly Borel probabilities on the separable Banach space $C_{\mathbb{T}}$ of the real valued continuous functions $\omega : \mathbb{T}\to \mathbb{R}$ defined on $\mathbb{T}$ vanishing at $0$ (i.e. $\omega(0)=0$), which is 
endowed with the norm of uniform convergence $\|.\|_{\infty}$, which is given by $\|\omega\|_{\mathbb{T}}:= \sup_{t\in\mathbb{T}}|\omega(t)|$, for each $\omega\in C_{\mathbb{T}}$; subsequently $\mathcal{B}(C_{\mathbb{T}})$ denotes the associated Borel $\sigma-$ fields on $C_{\mathbb{T}}$, while $M_1(C_
{\Delta})$ denotes the set of Borel probability measures on $C_{\mathbb{T}}$. Similarly, we denote by $C_{[0,1]}$ the set of continuous real valued functions  defined on the unit interval which vanish at $0$; it coincides with $C_{\mathbb{T}}$ when $\mathbb{T}=[0,1]$. In this paper, whenever $\mathbb{T}$ denotes a time scale as above,  to distinguish it from the so-called classical Wiener measure on $C_{[0,1]}$, we refer to the Wiener measure (\cite{WIENER1}, see also \cite{ITOMACKEAN}) on the space of functions $(C_{\mathbb{T}},\mathcal{B}(C_{\mathbb{T}}))$, which we denote by $\mu_{\mathbb{T}}\in M_1(C_{\mathbb{T}})$, as the \textit{sampled Wiener measure}. That is, $\mu_{\mathbb{T}}$ denotes the unique Borel probability measure on the measurable space $(C_{\mathbb{T}}, \mathcal{B}(C_{\mathbb{T}}))$ which is such that $\forall n\in [1, Card(\mathbb{T})-1]$ if $\mathbb{T}$ is a finite set, $\forall n\in \mathbb{N}$ otherwise, and for all $\{t_i : i \in \{0,...,n\}\}\subset \mathbb{T}$ such that $0=t_0<t_1<...<t_n$, setting $x_0:=0$,   the identity  $$\int_{C_{\mathbb{T}}} f(\omega(t_1),...,\omega(t_n))\mu_{\mathbb{T}}(d\omega) = \int_{\mathbb{R}^n} f(x_1,...,x_n) \frac{\exp\left(- \sum_{j=1}^n \frac{(x_{j}-x_{j-1})^2}{2(t_j-t_{j-1})}\right)}{(2\pi)^{\frac{n}{2}}(\Pi_{l=1}^n(t_l-t_{l-1}))^{\frac{1}{2}} } dx_1...dx_n,$$  holds $\forall f\in C_b(\mathbb{R}^n)$, continuous and bounded real valued function defined on $\mathbb{R}^n$. 
As it is well known, while the classical Wiener measure $\mu_{W}$ on $C_{[0,1]}$ is the law of the standard Brownian motion $(B_t)_{t\in [0,1]}$ starting from $0$, handled as a random function, the sampled Wiener measure $\mu_{\mathbb{T}}$ denotes similarly the law of the sampled Brownian motion $(B_t)_{t\in\mathbb{T}}$ handled as a random element of $C_{\mathbb{T}}$. Although the study of sampled stochastic processes constituted by the random variables $(B_t)_{t\in J}$ is usual, notably when $J\subset [0,1]$ is a finite set, taking $J$ to be a closed set brings enough structure to obtain astonishingly accurate results from the Hilger time scales tool box when $\mathbb{T}$ is neither an interval, nor a finite set, as it can be seen from early works as \cite{BOHNER2} and \cite{SANYAL}, which seems to have started investigations of those specific time scales tools within this stochastic context. Subsequently, we still denote by $\mu_{\mathbb{T}}$ the unique extension of this measure to the completed probability space $(C_{\mathbb{T}}, \mathcal{B}(C_{\mathbb{T}})^{\mu_{\mathbb{T}}})$.  Whenever $\mathbb{T}=[0,1]$, $\mu_{\mathbb{T}}$ coincide with the classical Wiener measure $\mu_{W}\in M_1(C_{[0,1]})$  (\cite{WIENER1}); notice that, by following \cite{BOHNER2}, the Guseinov measure on $(\mathbb{T},\mathcal{B}(\mathbb{T}))$ has been denoted by $\lambda_{\Delta}$, while whenever $E$ is a Polish space, and $\mathcal{B}(E)$ the associated Borel $\sigma-$ field on $E$, we denote by $M_1(E)$ the set of Borel probability measures on $(E,\mathcal{B}(E))$.

\subsection{Spaces of adapted processes and filtrations within this framework} 

Subsequent applications require the space
$$H_{\mathbb{T}}:= \left\{h \in C_{\mathbb{T}} : h(t) := \int_{[0,t)_{\mathbb{T}}} h^{\Delta}(s)  \lambda_{\Delta}(ds) , \forall t \in \mathbb{T}, \  h^{\Delta}\in L^2(\lambda_{\Delta}) \right\},$$ which is a separable Hilbert space, when it is endowed with the scalar product $$<h,k>_{H_{\mathbb{T}}}:= \int_{\mathbb{T}} h^{\Delta}(s) k^{\Delta}(s) \lambda_{\Delta} (ds), \ \forall h,k\in H_{\mathbb{T}}.$$ In the particular case where $\mathbb{T}= [0,1]$, we adopt the notation $H^1=H_{\mathbb{T}}$ from \cite{MALLIAVIN}; since in this case $\lambda_{\Delta}=\lambda$, it coincides with the classical Cameron-Martin space (\cite{MALLIAVIN}). The interest of this space within this paper will be justified \textit{a posteriori} in Section~\ref{ANALYTICBROWNIAN} as we identify the analytic structure of the $\mathbb{T}$ sampled Wiener measure.  Subsequently and until the end of the paper, we shall be essentially interested in sampled Brownian filtrations, more accurately in filtrations generated by a sampled Brownian motion. It motivates a specific notation within the framework of this paper. For $t\in \mathbb{T}$, we define $$\mathcal{F}_t := \sigma(W_s : s\in [0,t]_{\mathbb{T}}),$$ where $W_s :\omega\in C_{\mathbb{T}} \to \omega(s)\in \mathbb{R}$ denotes the evaluation function at any $s\in \mathbb{T}$. 
For each $s,t\in \mathbb{T}$ such that $s\leq t$, we have $\mathcal{F}_s \subset \mathcal{F}_t$, and $(\mathcal{F}_t)_{t\in \mathbb{T}}$ is  usually called the filtration generated by the evaluation process $(W_t)_{t\in \mathbb{T}}$, or the natural filtration of the canonical space.

Whenever $(\Omega, \mathcal{A}, \mathcal{P})$ denotes a complete probability space and $\mathcal{F}\subset \mathcal{A}$ is a $\sigma-$ field, we denote its $\mathbb{P}$ completion by $\mathcal{F}^{\mathbb{P}}$, which, by following an usual convention in stochastic analysis, is here the smallest sigma-field which contains both the sets of $\mathcal{F}$ and all the $\mathcal{P}$ negligible sets of $(\Omega, \mathcal{A}, \mathcal{P})$, not only those of $\mathcal{F}$.  Moreover, if $\mathbb{Q}$ denotes a probability measure which is defined on a not necessarily complete measurable space $(\Omega, \mathcal{F})$, then we still denote by $\mathbb{Q}$ its extension to the completed $\sigma-$ field $\mathcal{F}^{\mathbb{Q}}$.
In particular, whenever $\nu\in M_1(C_{\mathbb{T}})$ denotes a Borel probability measure on $C_{\mathbb{T}}$, $\mathcal{B}(C_{\mathbb{T}})^\nu$ denotes the $\nu-$ completion of $\mathcal{B}(C_{\mathbb{T}})$, while we still denote by $\nu$ the extension of this probability measure to $\mathcal{B}(C_{\mathbb{T}})^\nu$.

 To use tools of functional analysis, and of nonlinear analysis, we also need the following Hilbert spaces. We denote by $L^2(\mu_{\mathbb{T}}, H_{\mathbb{T}})$ the space of the $\mu_{\mathbb{T}}$ equivalence classes of Borel measurable functions $ u : \omega \in C_{\mathbb{T}} \to u(\omega) \in C_{\mathbb{T}}$ such that $$\int_{C_{\mathbb{T}}} \|u(\omega)\|_{H_{\mathbb{T}}}^2 \mu_{\mathbb{T}}(d\omega) <+\infty,$$ which are identified when they coincide outside some $\mu_{\mathbb{T}}-$ negligible set. To obtain suitable closed subsets of adapted processes, in this paper we use the following convention :
\begin{definition}
We denote by $L^2_a(\mu_{\mathbb{T}}, H_{\mathbb{T}})$, the subset of the $u:=$ $\int_{[0,.)_{\mathbb{T}}}u^{\Delta}_s$ $\lambda_{\Delta}(ds)$  $\in L^2(\mu_{\mathbb{T}}, H_{\mathbb{T}})$ which further satisfy both conditions :
\begin{enumerate}[(i)]
\item For all $t\in \mathbb{T}$, the random variable $u_t : \omega \in C_{\mathbb{T}} \to u_t(\omega) \in \mathbb{R}$ is $\mathcal{F}_t^{\mu_{\mathbb{T}}}-$ measurable.
\item For all $t\in \mathbb{T}_{rs}$ which is right-scattered, the random variable $u^{\Delta}_t : \omega \in C_{\mathbb{T}} \to u_t^{\Delta}(\omega) \in \mathbb{R}$ is $\mathcal{F}_t^{\mu_{\mathbb{T}}}-$ measurable.
\end{enumerate}
\end{definition}
In the above, we used the following notation : when $u\in L^2(\mu_{\mathbb{T}}, H_{\mathbb{T}})$ and $t\in \mathbb{T}$, $u_t:= W_t\circ u$, where $W_t: C_{\mathbb{T}} \to \mathbb{R}$ still denotes the evaluation map at $t\in \mathbb{T}$; depending on circumstances, $\circ$ may denote either the usual pullback of functions, or of equivalence classes of measurable functions whenever its is uniquely defined.  Notice that whenever  $\mathbb{T}=[0,1]$, the previous definition reduces to the condition $(i)$ which is usual in continuous time.  The previous definition is notably motivated by the following Proposition, which easily follows from the Jensen inequality :

\begin{proposition}
The space $L^2_a(\mu_{\mathbb{T}}, H_{\mathbb{T}})$ is a closed linear subspace of the Hilbert space $L^2(\mu_{\mathbb{T}}, H_{\mathbb{T}})$, for the strong topology.
\end{proposition}

To control specific absolute continuity assumptions below, we will use the subset  $L^{\infty}_a(\mu_{\mathbb{T}},H_{\mathbb{T}})$ of the $h\in L^2_a(\mu_{\mathbb{T}}
,H_{\mathbb{T}})$ such that there exists a $c>0$ which satisfies $$\|\theta\|_{H_{\mathbb{T}}}< c, \mu_{\mathbb{T}}-a.s..$$

\section{The analytic structure of the sampled Brownian motion}
\label{ANALYTICBROWNIAN}

\subsection{The abstract Wiener space $(H_{\mathbb{T}}, C_{\mathbb{T}}, \mu_{\mathbb{T}})$}

An abstract Wiener space is the product which is usually obtained when we apply the Gross procedure (\cite{gr65}) to a separable Hilbert space. We refer to \cite{KUO}, \cite{MALLIAVIN},  \cite{STROOCK}, \cite{USTUNELBOOKZ} for much details and several approaches to the construction of such spaces. Recall that abstract Wiener spaces of L. Gross provide the analytic structure of Wiener measures on Banach spaces, and provide in particular an analytic point of view on the classical Wiener measure (\cite{WIENER1}).

\begin{proposition} \label{JDEF}
Let $J : C_{[0,1]} \to C_{\mathbb{T}}$ be the linear continuous function which to any $\omega\in C_{[0,1]}$ associates its restriction $J(\omega) \in C_{\mathbb{T}}$ to $\mathbb{T}$,  that is $J_t(\omega):= \omega(t)$, $\forall t\in \mathbb{T}$. Then, the restriction $J|_{H^1}$ of the linear continuous operator $J$ to $H^1$ has range $H_{\mathbb{T}}$, and admits an adjoint operator $$J|_{H^1}^\star : H_{\mathbb{T}} \to H^1,$$ with domain the whole $H_{\mathbb{T}}$, which maps $H_{\mathbb{T}}$ isometrically into its range $J|_{H^1}^\star(H_{\mathbb{T}})$.
\end{proposition}
\begin{proof}
Given $\widetilde{h}:=\int_0^. \dot{\widetilde{h}}_s ds \in H^1$, define \begin{equation} \label{Jhdensity} \widetilde{h}^{\Delta}_t := \dot{\widetilde{h}}_t 1_{\mathbb{T}\setminus \mathbb{T}_{rs}}(t) + \sum_{t_i\in \mathbb{T}_{rs}} \left(\frac{\widetilde{h}_{\bar{\sigma}(t_i)} - \widetilde{h}_{t_i} }{\bar{\sigma}(t_i)-t_i}\right) 1_{\{t_i\}}(t), \ \forall t\in \mathbb{T},\end{equation} From the definitions, the Lebesgue decomposition formula~(\ref{Hilgerdecomp}) of the Guseinov measure $\lambda_{\Delta}$ easily yields $$\widetilde{h}(t)= \int_{[0,t)_{\mathbb{T}}} \widetilde{h}^{\Delta}_s \lambda_{\Delta}(ds), \ \forall t\in \mathbb{T},$$ while together with~(\ref{Hilgerdecomp}), Jensen's inequality ensures that $\widetilde{h}^{\Delta}\in \mathcal{L}^2(\lambda_{\Delta})$. This entails the inclusion $J(H^1)\subset H_{\mathbb{T}}$. Conversely, given $h:=\int_{[0,.)} h^{\Delta}_s \lambda_{\Delta}(ds) \in H_{\mathbb{T}}$, define $$g_t(h):= \int_0^t \dot{\widetilde{h}}_s \lambda(ds), \ \forall t\in [0,1]$$ where we now set $$\dot{\widetilde{h}}_s := (h^{\Delta} 1_{\mathbb{T}})(s) + \sum_{t_i \in \mathbb{T}_{rs}} h^{\Delta}(t_i) 1_{]t_i, \bar{\sigma}(t_i)]}(s), \ \forall s\in [0,1],$$ which easily entails with~(\ref{Hilgerdecomp}) that $\|g(h)\|_{H^1}=\|h\|_{H_{\mathbb{T}}}<+\infty$, and that $J\circ g(h)= h$, that is $g_t(h)=h_t$, $\forall t\in \mathbb{T}$. This ensures that $H_{\mathbb{T}}\subset J(H^1)$, so that we finally obtain $H_{\mathbb{T}}= J(H^1)$.  Hence, the restriction $J|_{H^1} : H^1 \to H_{\mathbb{T}}$ of the linear operator $J$ to $H^1$ has \textit{a fortiori} a dense domain. In this case, from standard results of functional analysis (see \cite{BREZIS}), it has an adjoint operator $J|_{H^1}^\star$. Its domain coincides with $H_{\mathbb{T}}$, since together with the decomposition formula~(\ref{Hilgerdecomp}), Jensen's inequality and the Cauchy-Schwarz inequality easily yield $$ |<h, J(\widetilde{h})>_{H_{\mathbb{T}}}| \leq 2\|h\|_{H_{\mathbb{T}}}\|\widetilde{h}\|_{H^1}, \ \forall h\in H_{\mathbb{T}}, \ \widetilde{h}\in H^1,$$ from~(\ref{Jhdensity}).  Therefore, $J|_{H^1}^\star : H_{\mathbb{T}}\to H^1$ is well defined and satisfies \begin{equation} \label{adjointcheck1} <\widetilde{h},  J|_{H^1}^\star[h]>_{H^1}= <J(\widetilde{h}), h>_{H_{\mathbb{T}}}, \end{equation} $\forall h\in H_{\mathbb{T}},$ $\forall \widetilde{h} \in H^1.$  Then, it easily follows from the decomposition formula~(\ref{Hilgerdecomp}), that given $h\in H_{\mathbb{T}}$, $g(h)\in H^1$ defined above satisfies~(\ref{adjointcheck1}), $\forall \widetilde{h}\in H^1$. Hence, we have identified $J|_{H^1}^{\star}$; $J|_{H^1}^\star(h)= g(h)$, $\forall h\in H_{\mathbb{T}}$, so that the isometry property follows from $\|g(h)\|_{H^1}= \|h\|_{H_{\mathbb{T}}}$, $\forall h\in H_{\mathbb{T}}$. 
\end{proof}

\begin{proposition}  \label{abstractWiener}
The triplet  $(H_{\mathbb{T}}, C_{\mathbb{T}}, \mu_{\mathbb{T}})$ is an abstract Wiener space. \end{proposition}
\begin{proof}
Since the continuous injection of the separable Hilbert space $H_{\mathbb{T}}$ in the separable Banach space $C_{\mathbb{T}}$ is easily checked to have a dense range by standard tools of analysis, it is enough to identify the characteristic function of $\mu_{\mathbb{T}}$.  Recall that since $H_{\mathbb{T}}$ (resp. $H^1$) is a separable Hilbert space whose injection in $C_{\mathbb{T}}$ (resp. in $W:=C_{[0,1]}$) is both dense and continuous (and also compact), we already know that for any $l\in C_{\mathbb{T}}^\star$ (resp. $\widetilde{l}\in C_{[0,1]}^\star$), the topological dual of $C_{\mathbb{T}}$ (resp. of $C_{[0,1]}$), there exists a unique $j^{\mathbb{T}}(l)\in H_{\mathbb{T}}$ (resp. $j^{W}(\widetilde{l})\in H^1$) such that \begin{equation} \label{jstartmp} <l, h>_{C_{\mathbb{T}}^\star, C_{\mathbb{T}}} = <j^{\mathbb{T}}(l), h>_{H_{\mathbb{T}}}, \ \forall h\in H_{\mathbb{T}},\end{equation} respectively  \begin{equation} <\widetilde{l}, \widetilde{h}>_{C_{[0,1]}^\star, C_{[0,1]}} = <j^{W}(\widetilde{l}), \widetilde{h}>_{H^1}, \ \forall \widetilde{h}\in H^1.\end{equation}  For $l\in C_{\mathbb{T}}^\star$, and $\widetilde{h}\in H^1$, since $J: C_{[0,1]}\to C_{\mathbb{T}}$ is both linear and continuous, and $J(C_{[0,1]})= C_{\mathbb{T}}$,  we have $l\circ J\in W^{\star}$,
 so that we obtain $$<j^{W}(l\circ J), \widetilde{h}>_{H^1}= <l\circ J, \widetilde{h}>_{C_{[0,1]}^\star, C_{[0,1]}},$$ from the definition of $j^W$. On the other hand, since $J\circ \widetilde{h}\in H_{\mathbb{T}}\subset C_{\mathbb{T}}$ follows from Proposition~\ref{JDEF}, the definition of $j^{\mathbb{T}}$ and~(\ref{jstartmp})  yield $$<l\circ J, \widetilde{h}>_{C_{[0,1]}^\star, C_{[0,1]}} = <l, J(\widetilde{h})>_{C_{\mathbb{T}}^\star, C_{\mathbb{T}}} = <j^{\mathbb{T}}(l), J|_{H^1}(\widetilde{h})>_{H_{\mathbb{T}}}.$$ Together with the definition of the adjoint we finally obtain $$<j^{W}(l\circ J), \widetilde{h}>_{H^1} = <J|_{H^1}^\star(j^{\mathbb{T}}(l)), \widetilde{h}>_{H^1},$$ and since it holds for all $\widetilde{h}\in H^1$, we obtain \begin{equation} j^{W}(l\circ J) =J|_{H^1}^\star(j^{\mathbb{T}}(l)). \label{identif1} \end{equation}

Let $l\in C_{\mathbb{T}}^\star$, since $\mu_{\mathbb{T}}=J_{\star}\mu_W$ coincide with the pushforward of the Borel probability measure $\mu_W$ by the measurable function $J$, from the definition of the classical Wiener measure, from~(\ref{identif1}), and from the isometric property established by Proposition~\ref{JDEF}, we obtain  \begin{eqnarray*}  \int_{C_{\mathbb{T}}}\exp(i <l,\omega>_{C_{\mathbb{T}}^\star, C_{\mathbb{T}}})  \mu_{\mathbb{T}}(d\omega) & = & \int_{C_{\mathbb{T}}}\exp(i <l,\omega>_{C_{\mathbb{T}}^\star, C_{\mathbb{T}}}) J_{\star}{\mu_W}(d\omega) \\ & = &  \int_{W}\exp(i <l,J(\omega)>_{C_{\mathbb{T}}^\star, C_{\mathbb{T}}}) {\mu_W}(d\omega)  \\ & = &   \int_{W}\exp(i <l\circ J,\omega>_{C_{[0,1]}^\star, C_{[0,1]}}) {\mu_W}(d\omega)  \\ & = & \exp\left(- \frac{\|j^{W}(l\circ J)\|_{H^1}^2}{2}\right) \\ & = &  \exp\left(- \frac{\|J|_{H^1}^\star (j^{\mathbb{T}}(l))\|_{H^1}^2}{2}\right) \\ & = &  \exp\left(- \frac{\| j^{\mathbb{T}}(l)\|_{H^\Delta }^2}{2}\right) , \end{eqnarray*}  which proves the result; we used the notation $i:=\sqrt{-1}$. 
 \end{proof}

\subsection{The sampled Wiener measure as an abstract Wiener space}

From Proposition~\ref{abstractWiener}, the analytic structure of the sampled Brownian motion is provided by the abstract Wiener space $(H_{\mathbb{T}},C_{\mathbb{T}},\mu_{\mathbb{T}})$  (see \cite{gr65}, \cite{KUO}, \cite{STROOCK}), where $H_{\mathbb{T}}$ is expressed by integrals of Hilger's time scales analysis. We recall here the main consequences useful to perform explicit computations below, in particular to identify the Paley-Wiener integrals.  First, the injection $$C_{\mathbb{T}}^\star \hookrightarrow_{j^{\mathbb{T}}} H_{\mathbb{T}},$$ of the topological dual $C_{\mathbb{T}}^\star$ of $C_{\mathbb{T}}$ into $H_{\mathbb{T}}$ is dense, and  when $l\in C_{\mathbb{T}}^\star$, its image $j^{\mathbb{T}}(l)\in H_{\mathbb{T}}$ is the unique element of $H_{\mathbb{T}}$ such that \begin{equation} \label{jTident} <j^{\mathbb{T}}(l),  h >_{H_{\mathbb{T}}}  = <l, h>_{C_{\mathbb{T}}^\star, C_{\mathbb{T}}}, \ \forall h\in H_{\mathbb{T}}. \end{equation} The latter is of help to identify the characteristic function $\phi_{\mu_{\mathbb{T}}} : C_{\mathbb{T}}^\star \to \mathbb{R}$ of  $\mu_{\mathbb{T}}$ which is given by \begin{equation} \phi_{\mu_{\mathbb{T}}}(l) = \exp\left( -\frac{\|j^{\mathbb{T}}(l)\|_{H_{\mathbb{T}}}^2}{2}\right),  \ \forall l\in C_{\mathbb{T}}^\star, \label{caracteristicdelta} \end{equation} where $$\phi_{\mu_{\mathbb{T}}}(l) := \int_{C_{\mathbb{T}}} \exp( i <l,\omega>_{C_{\mathbb{T}}^\star, C_{\mathbb{T}}} ) \mu_{\mathbb{T}}(d\omega),  \ \forall l\in C_{\mathbb{T}}^\star.$$  In particular, as expected, given $s,t\in \mathbb{T}$ with $s\leq t$, the latter yields $$E_{\mu_{\mathbb{T}}}\left[\exp(i \lambda (W_t-W_s))\right] = \exp\left( -(t-s) \frac{\lambda^2}{2}\right), \ \forall \lambda\in \mathbb{R}$$ as a particular case, which follows by identifying $j^{\mathbb{T}}(l)$ from~(\ref{jTident}) with $l:=W_t-W_s$, $s,t\in \mathbb{T}$, $s\leq t$, and from the linearity of $j^{\mathbb{T}}$; whenever $f\in \mathcal{L}^1(\mu_{\mathbb{T}})$, $E_{\mu_{\mathbb{T}}}[f]$ denotes the mathematical expectation of $f$, that is $E_{\mu_{\mathbb{T}}}[f]:=\int_{C_{\mathbb{T}}} f(\omega) \mu_{\mathbb{T}}(d\omega)$.  Recall that a stochastic process $(X_t)_{t\in \mathbb{T}}$ defined on a complete filtered probability space $(\Omega,\mathcal{A},$ $(\mathcal{A}_t)_{t\in \mathbb{T}},$ $\mathcal{P})$,  is said to be $(\mathcal{A}_t)_{t\in\mathbb{T}}$ adapted, if $X_t$ is $\mathcal{A}_t-$ measurable, $\forall t\in \mathbb{T}$. Furthermore, to shorten the statements subsequently, and keeping a consistency with \cite{BOHNER2}, \cite{SANYAL},  given  a sampled Brownian motion $(B_t)_{t\in \mathbb{T}}$ on a complete filtered space $(\Omega, \mathcal{A}, (\mathcal{A}_t)_{t\in \mathbb{T}}, \mathcal{P})$, $(B_t)_{t\in \mathbb{T}}$ is  called an $(\mathcal{A}_t)_{t\in \mathbb{T}}-$ sampled Brownian motion, if and only if,  the stochastic process $(B_t)_{t\in \mathbb{T}}$ is $(\mathcal{A}_t)_{t\in\mathbb{T}}-$ adapted and  $$E_{\mathcal{P}}\left[\exp(i \lambda (B_t-B_s))\middle| \mathcal{A}_s\right] = \exp\left( -(t-s) \frac{\lambda^2}{2}\right), \ \mathcal{P}-a.s.,$$ $\forall \lambda\in \mathbb{R}$ and $s<t$, $s,t\in \mathbb{T}$, where the left hand term denotes a conditional expectation with respect to the complete $\sigma-$ field $\mathcal{A}_s$ ; see \cite{I-W}.

\subsection{The Paley-Wiener integrals}

Denoting by $$\mathcal{I} :  j^{\mathbb{T}}(C_{\mathbb{T}}^\star) \to L^2(\mu_{\mathbb{T}})$$ the injection which to a $h:=j^{\mathbb{T}}(l)\in j(C_{\mathbb{T}}^\star)\subset H_{\mathbb{T}}$ for some $l\in C_{\mathbb{T}}^\star$,  associates the corresponding $
\mu_{\mathbb{T}}$ equivalence class of measurable real valued functions on $C_{\mathbb{T}}$, so that  $$\mathcal{I}(j^{\mathbb{T}}(l))(\omega) = <l,\omega>_{C_{\mathbb{T}^\star}, C_{\mathbb{T}}}, \mu_{\mathbb{T}}-a.s., \ \forall l\in C_{\mathbb{T}}^\star, $$ recall that we obtain a linear isometry since $$ \| \mathcal{I}( h)\|_{L^2(\mu_{\mathbb{T}})}= \|h\|_{H_
{\Delta}},$$ $\forall h \in j^{\mathbb{T}}(C_{\mathbb{T}}^\star)$; the latter easily follows from~(\ref{caracteristicdelta}) since by linearity of $j^{\mathbb{T}}$,  standard results on the Fourier transform entail that $\mathcal{I}(h)$ is a real-valued centered gaussian random variable on the probability space $(C_{\mathbb{T}}, \mathcal{B}(C_{\mathbb{T}})^{\mu_{\mathbb{T}}}, \mu_{\mathbb{T}})$, with variance $\|h\|_{H_{\mathbb{T}}}^2$, whenever $h=j^{\mathbb{T}}(l)$, $l\in C_{\mathbb{T}}^\star$. Therefore, as an isometry with dense domain it extends uniquely to a function  $\mathcal{I} : H_{\mathbb{T}} \to L^2(\mu_{\mathbb{T}}),$ the 
 Paley-Wiener integral (see \cite{STROOCK}), and we shall use the notation $$\int_{\mathbb{T}} h^
{\Delta}(s) d_{\Delta} W_s := \mathcal{I}(h), \ \mu_{\mathbb{T}}-a.s. \ \forall h\in H_{\mathbb{T}};$$ when $
\mathbb{T}=[0,1]$ it reduces to the classical Paley-Wiener stochastic integral. Within this convention, identifying $j^{\mathbb{T}}(l)$ from~(\ref{jTident}) with now $l:=W_{\bar{\sigma}(t)}-W_{\bar{\sigma}(s)}$, $s,t\in \mathbb{T}$, $s\leq t$,  from the definitions we obtain the following
relation which allows to check that the values of the latter stochastic integrals coincide with those of the $\Delta-$ stochastic integral from the construction of \cite{BOHNER2}, which therefore boils down to an abstract Paley-Wiener integral,  when it is applied to usual indicators :

\begin{proposition}
Let $s,t\in \mathbb{T},$ be such that $s\leq t$, then we have 
$$\int_{\mathbb{T}} 1_{]s,t]}(u) d_{\Delta}W_u = W_{t + \bar{\mu}_g(t)} - W_{s+ \bar{\mu}_g(s)}, \ \mu_{\mathbb{T}}-a.s., \  
$$ where $\bar{\mu}_g$ still denotes the graininess function of the time-scale $\mathbb{T}$. 
\end{proposition}

\section{Quasi-invariance, a Kailath-Zakai absolute continuity lemma}
\label{CML}
\subsection{Quasi-invariance}

In this section, we introduce a lemma which will be used subsequently to ensure the definition of the function whose fixed points will be under consideration. This lemma~\ref{Lemmabs} is related to the behaviour of $\mu_{\mathbb{T}}$ null sets within a class of adapted deterministic transports of measures.  As a warm-up,  given $\widetilde{\omega} \in C_{\mathbb{T}}$, define the translation operator $$\tau_{\widetilde{\omega}} : \omega \in C_{\mathbb{T}} \to \omega + \widetilde{\omega}\in C_{\mathbb{T}},$$ which is continuous, and thus Borel measurable. As it is pointed out with much details in \cite{STROOCK2}, due to the Sudakov theorem (\cite{SUDAKOV}) a basic question in such contexts is to determine whether  the pushforward measure ${\tau_{\widetilde{\omega}}}_\star \mu_{\mathbb{T}}\in M_1(C_{\mathbb{T}})$ of $\mu_{\mathbb{T}}$ by $\tau_{\widetilde{\omega}}$ is equivalent to $\mu_{\mathbb{T}}\in M_1(C_{\mathbb{T}})$, i.e. ${\tau_{\widetilde{\omega}}}_\star \mu_{\mathbb{T}}(A)=0$, if and only if, $\mu_{\mathbb{T}}(A)=0$, $\forall A\in \mathcal{B}(C_{\mathbb{T}})$; ${\tau_{\widetilde{\omega}}}_\star \mu_{\mathbb{T}}(A):=\mu_{\mathbb{T}}( \tau_{\widetilde{\omega}}^{-1}(A))$, $\forall A\in \mathcal{B}(C_{\mathbb{T}})$. Although as soon as $\mathbb{T}$ is such that $C_{\mathbb{T}}$ is infinite dimensional, not all $\widetilde{\omega}$ induce a translation which preserve the equivalence classes of $\mu_{\mathbb{T}}$, the Cameron-Martin-Segal theorem, from works of R.H. Cameron and W.T. Martin  (\cite{CAMMARTIN1},\cite{ CAMMARTIN2}), completed by I. Segal  (\cite{SEGAL}), which still holds on abstract Wiener spaces, applies to the present framework; we refer notably to \cite{MALLIAVIN} and \cite{STROOCK2} for references with a clear statement of this result.

\begin{theorem} \label{HCMtheorem} ($H_{\mathbb{T}}$-Cameron-Martin-Segal theorem.) 
For any $h\in C_{\mathbb{T}}$, denoting by $\tau_h : \omega \in C_{\mathbb{T}} \to \omega + h \in C_{\mathbb{T}}$, the probability measure ${\tau_{h}}_{\star}\mu_{\mathbb{T}}$ is absolutely continuous with respect to the sampled Wiener measure $\mu_{\mathbb{T}}$, if and only if, $h\in H_{\mathbb{T}}$ ; otherwise it is singular with respect to $\mu_{\mathbb{T}}$. Moreover, for any $h\in H_{\Delta}$,  ${\tau_{h}}_{\star}\mu_{\mathbb{T}}$ is also equivalent to $\mu_{\mathbb{T}}$, with a Radon-Nikodym derivative $\frac{d{\tau_{h}}_{\star}\mu_{\mathbb{T}}}{d \mu_{\mathbb{T}}}\in L^1(\mu_{\mathbb{T}})$ which is given by $$\frac{d{\tau_{h}}_{\star}\mu_{\mathbb{T}}}{d \mu_{\mathbb{T}}} = \exp\left(\int_{\mathbb{T}} h^{\Delta}(s) d_{\Delta}W_s -\frac{1}{2}\int_{\mathbb{T}} |h^{\Delta}(s)|^2 \lambda_\Delta(ds) \right), \ \mu_{\mathbb{T}}-a.e.$$
\end{theorem}

\subsection{A Kailath-Zakai absolute continuity lemma}

We are now ready to fix further results required for our purposes. Given a Borel measurable function $u : C_{\mathbb{T}} \to C_{\mathbb{T}}$, we now denote \begin{equation} \tau_{u} : \omega \in C_{\mathbb{T}} \to \omega + u(\omega) \in C_{\mathbb{T}}, \label{tauhdef} \end{equation} which defines a deterministic transport plan which maps $\mu_{\mathbb{T}}\in M_1(C_{\mathbb{T}})$ to the pushforward measure ${\tau_u}_{\star}\mu_{\mathbb{T}}\in M_1(C_{\mathbb{T}})$ of $\mu_{\mathbb{T}}\in M_1(C_{\mathbb{T}})$ by the Borel measurable function $\tau_u$. Subsequently, we use the same notation whenever $u\in L^2_a(\mu_{\mathbb{T}}, H_{\mathbb{T}})$; if necessary, for direct images of probability measures by equivalence classes of measurable functions, see  \cite{MALLIAVIN2}. The following Lemma~\ref{Lemmabs} extends to $\mathbb{T}$ a sufficient condition to obtain the absolute continuity (see \cite{HALMOS}) of ${\tau_u}_{\star}\mu_{\mathbb{T}}$ with respect to $\mu_{\mathbb{T}}$, which is denoted by ${\tau_u}_{\star}\mu_{\mathbb{T}} << \mu_{\mathbb{T}}$.  Recall that the notation of  $L^\infty_a(\mu_{\mathbb{T}}, H_{\mathbb{T}})$ has also been given in the previous sections. Although the statement could be extended (see \cite{KailathZakai71} when $\mathbb{T}=[0,1]$), the following lemma will be sufficient to state our main result.

\begin{lemma} \label{Lemmabs}
For any $u\in L^2_a(\mu_{\mathbb{T}}, H_{\mathbb{T}})$, we have $${\tau_{u}}_{\star} \mu_{\mathbb{T}} << \mu_{\mathbb{T}},$$ that is, the Borel probability measure ${\tau_{u}}_{\star} \mu_{\mathbb{T}}$ is absolutely continuous  with respect to the sampled Wiener measure $\mu_{\mathbb{T}}$. Moreover, if $u\in L^\infty_a(\mu_{\mathbb{T}}, H_{\mathbb{T}})$, then ${\tau_{u}}_{\star} \mu_{\mathbb{T}} \sim \mu_{\mathbb{T}},$ i.e. both probability measures are equivalent. 
\end{lemma} 
\begin{proof}
Still denote by $\mu_W$ the standard classical Wiener measure which is a Borel probability measure on the space $C_{[0,1]}$ of continuous real valued functions on the unit interval, and by $H^1$ the classical Cameron-Martin space; $\mu_{W}$ (resp. $H^1$) coincides with $\mu_{\mathbb{T}}$ (resp. with $H_{\mathbb{T}}$) when we take $\mathbb{T}=[0,1]$. 
From Proposition~\ref{JDEF} if $u\in L^2_a(\mu_{\mathbb{T}}, H_{\mathbb{T}})$, by setting $\widetilde{u}(\omega) := J|_{H_1}^\star(u(\omega))$, $\mu_{\mathbb{T}}-$ a.e.,  the isometric embedding of $H_{\mathbb{T}}$ into $H^1$  entails $E_{\mu_{\mathbb{T}}}[\|\widetilde{u}\|_{H^1}^2] < +\infty$. Moreover, from the form of $J|_{H_1}^\star$, which is identified in the proof of  Proposition~\ref{JDEF}, and from the definition of $L^2_a
(\mu_{\mathbb{T}}, H_{\mathbb{T}})$, given $t\in [0,1]$, we first obtain that $\widetilde{u}_t$ is $\mathcal{F}_t^{\mu_{\mathbb{T}}}-$ measurable if $t\in \mathbb{T}$, while it is $\mathcal{F}_{\bar
{\rho}(t)}^{\mu_{\mathbb{T}}}-$ measurable if $t\in [0,1]\setminus \mathbb{T}$, where we have set $\widetilde{u}_t:=\widetilde{W}_t\circ \widetilde{u}$, and where $\widetilde{W}_t : \widetilde{\omega} \in C_{[0,1]} \to \widetilde{\omega}(t) \in \mathbb{R}$ denotes here the evaluation map of $C_{[0,1]}$ at $t\in [0,1]$; recall that $\widetilde{u}$ takes its values in $H^1\subset C_{[0,1]}$. As a consequence, for all $t\in [0,1]$, $\widetilde{u}_t\circ J$ is $\widetilde{\mathcal{F}}_t^{\mu_W}-$ measurable, where we also used that,  from the definition of $J$, it follows that $J^{-1}
(\mathcal{F}_{s}^{\mu_{\mathbb{T}}})\subset \widetilde{\mathcal{F}}_s^{\mu_W}$, if $s\in \mathbb{T}$, and $J^{-1}(\mathcal{F}_{\bar{\rho}(s)}^{\mu_{\mathbb{T}}})\subset \widetilde{\mathcal{F}}_{s}^{\mu_W}$, if $s\in [0,1]\setminus 
\mathbb{T}$, where $\widetilde{\mathcal{F}}_t^{\mu_W}$ denotes the $\sigma-$ field on $C_{[0,1]}$ which is the $\mu_W$ completion of  $\widetilde{\mathcal{F}}_t:= \sigma(\widetilde{W}_s, s\in [0,t])$, $\forall t\in [0,1]$, while $\mu_{\mathbb{T}}= J_{\star}\mu_W$. Therefore, $\widetilde{u}\circ J \in L^2_a(\mu_W, H^1)$, which according to the previous notation denotes the set of equivalence classes of Borel measurable functions $h : C_{[0,1]}\to H^1$ such that $\int_{C_{[0,1]}} \| h(\widetilde{\omega})\|^2_{H^1} \mu_W(d\widetilde{\omega}) <+\infty$, which are identified when they coincide outside a $\mu_W$ negligible set, and are moreover such that $h_t:= \widetilde{W}_t\circ h$ is $\mathcal{F}_t^{\mu_W}-$ measurable, $\forall t\in [0,1]$.  Furthermore, from the construction of $\widetilde{u}$, the proof of~Proposition~\ref{JDEF} ensures that \begin
{equation} u\circ J = J \circ \widetilde{u}\circ J, \ \mu_W-a.s. \label{isomcomposeLM} \end{equation}  Hence, since $\mu_{\mathbb{T}}=J_{\star}\mu_W$,  from the linearity of $J$ and from~(\ref{isomcomposeLM}), we obtain
\begin{eqnarray*}
(I_{C_{\mathbb{T}}}+ u)_{\star}\mu_{\mathbb{T}}  & = & (I_{C_{\mathbb{T}}} + u)_{\star}J_\star \mu_{W} \\ & = & (J+ u\circ J)_{\star}\mu_W \\ & = & (I_{C_{\mathbb{T}}} \circ J+  J \circ \widetilde{u}\circ J)_{\star}\mu_W \\ & = & (J\circ I_{C_{[0,1]}} + J \circ \widetilde{u}\circ J)_{\star}\mu_W  \\ & = & J_\star ((I_{C_{[0,1]}} + \widetilde{u}\circ J)_\star\mu_W). \end{eqnarray*} Since $ \widetilde{u}\circ J \in L^2_a(\mu_W, H^1)$,  Theorem 1 of \cite{KailathZakai71} ensures that $ \nu:=(I_{C_{[0,1]}} + \widetilde{u}\circ J)_\star\mu_W<<\mu_W$, i.e. absolutely continuous. Hence, we conclude that $$(I_{C_{\mathbb{T}}}+ u)_{\star}\mu_{\mathbb{T}} = J_\star \nu << J_{\star} \mu_W = \mu_{\mathbb{T}},$$ which is the first part of the result. Similarly, if we further assume that $u\in L_a^\infty(\mu_{\mathbb{T}},$ $H_{\mathbb{T}})$, since $J_{\star}\mu_W=\mu_{\mathbb{T}}$, from the isometric property of Proposition~\ref{JDEF}, there exists a $c>0$ such that $$\|\widetilde{u}\circ J\|_{H^1} <c, \ \mu_W-a.s..$$ Therefore, the Novikov condition (see \cite{I-W}, \cite{NOVIKOV}) ensures here that  the Cameron-Martin-Maruyama-Girsanov theorem applies (see \cite{GIRSANOV}, \cite{I-W}, \cite{Maruyama55}),
and from this we obtain $$(I_{C_{[0,1]}}+ \widetilde{u}\circ J)_{\star}\mu_{W} \sim \mu_W;$$ for instance see  \cite{KailathZakai71}, or \cite{USTUNELBOOKZ}. Therefore, we conclude that $$(I_{C_{\mathbb{T}}}+ u)_{\star}\mu_{\mathbb{T}} =  J_\star ((I_{C_{[0,1]}} + \widetilde{u}\circ J)_\star\mu_W) \sim J_{\star} \mu_W =\mu_{\mathbb{T}}.$$  
\end{proof}

\section{A sufficient condition of existence from the Frank H. Clarke contraction criteria}
\label{SCC}

In the following theorem, given a measurable function $X : C_{\mathbb{T}} \to C_{\mathbb{T}}$, and $\nu\in M_1(C_{\mathbb{T}})$, we say that $X$ is a generating sampled Brownian motion on 
 $(C_{\mathbb{T}}, $ $\mathcal{B}(C_{\mathbb{T}})^\nu$ $,$ $\nu)$, if and only if, $(X_t)_{t\in \mathbb{T}}$ is a sampled Brownian motion on the probability space $(C_{\mathbb{T}}, \mathcal{B}(C_{\mathbb{T}})^\nu,$ $ \nu)$, and $$\mathcal{G}_t^X = \mathcal{F}_t^\nu, \ \forall t\in \mathbb{T},$$ where $\mathcal{F}_t^{\nu}:=\sigma(W_s, s\in [0,t]_{\mathbb{T}})^\nu$, where $\mathcal{G}_t^{X}:=\sigma(X_s,  s\in [0,t]_{\mathbb{T}})^\nu$, where $X_t:= W_t\circ X$, $\forall t\in \mathbb{T}$, and where $W_t : \omega\in C_{\mathbb{T}}\to \omega(t)\in \mathbb{R}$ still denotes the evaluation function at any $t\in \mathbb{T}$. Notice that whenever $\mathbb{T}=[0,1]$, we fall into a classical generating Brownian motion. Given $X:= I_{C_{\mathbb{T}}}+\theta$, where $\theta\in L^\infty_a(\mu_{\mathbb{T}},H_{\mathbb{T}})$ and $h\in L^2_a(\mu_{\mathbb{T}},H_{\mathbb{T}})$, Lemma~\ref{Lemmabs} ensures that  $$\kappa_X[h] :=  I_{C_{\mathbb{T}}}- X\circ \tau_h= - h-\theta\circ(I_{C_{\mathbb{T}}}+h)$$ is a well defined element of $L^2(\mu_{\mathbb{T}},H_{\mathbb{T}})$, which is easily checked to be also an element of the closed linear subspace $L^2_a(\mu_{\mathbb{T}},H_{\mathbb{T}})$. Therefore, under those hypothesis, we obtain a possibly nonlinear operator \begin{equation} \kappa_{X} : h \in L^2_a(\mu_{\mathbb{T}}, H_{\mathbb{T}}) \to  I_{C_{\mathbb{T}} }-X\circ \tau_h \in L^2_a(\mu_{\mathbb{T}}, H_{\mathbb{T}}), \label{kappadef} \end{equation} where $\tau_h := I_{C_{\mathbb{T}}}+ h$, $\forall h\in L^2_a(\mu_{\mathbb{T}},H_{\mathbb{T}})$. Subsequently, the operator $\kappa_X$ and the set $L^{\infty}_a$ $(\mu_{\mathbb{T}},$ $H_{\mathbb{T}})$ enlighten the notation in the statement of Theorem~\ref{theorem2}.  The proof of the latter is built on an application of the fixed point theorem of  F. H. Clarke, which extends the fundamental Banach contraction result, and allows us to obtain, under fairly general sufficient conditions, the existence of a Borel probability measure $\nu\in M_1(C_{\mathbb{T}})$, which turns a process $(X_t)_{t\in \mathbb{T}}$, which is defined on the canonical space, into a generating sampled Brownian motion.

\begin{theorem} \label{theorem2}
Let  $X:= I_{C_{\mathbb{T}}} + \theta$, where $\theta\in L^\infty_a(\mu_{\mathbb{T}}, H_{\mathbb{T}})$.  Further denote by $\kappa_X$ the map which is associated to $X$ by~(\ref{kappadef}). If we further assume that $X$ meets the following hypothesis 
\begin{enumerate}[(i)]
\item the operator $\kappa_X$ is continuous.
\item  $\exists K\in [0,1[$ such that $\forall h\in L^2_a(\mu_{\mathbb{T}}, H_{\mathbb{T}})$, $\exists \epsilon  \in ]0,1]$ which satisfy  $$\int_{C_{\mathbb{T}} } \|X\circ \tau_{h+ \epsilon \kappa_X[h]}- X\circ \tau_h - \epsilon \kappa_X[h]\|_{H_{\mathbb{T}}}^2 d\mu_\Delta  \leq K \epsilon^2 \int_{C_{\mathbb{T}}} \|\kappa_X[h]\|_{H_{\mathbb{T}}}^2 d\mu_{\Delta}.$$
\end{enumerate}
Then, there exists a Borel probability measure $\nu \in M_1(C_{\mathbb{T}})$, equivalent to $\mu_{\mathbb{T}}$, $\nu \sim \mu_{\mathbb{T}}$, which is such that $(X_t)_{t\in \mathbb{T}}$ is a generating sampled Brownian motion on $(C_{\mathbb{T}},$ $\mathcal{B}(C_{\mathbb{T}})^\nu,$ $\nu)$. In particular, we have the coincidence of  $\sigma$-fields \begin{equation}  \mathcal{G}_t^X  = \mathcal{F}_t^\nu, \ \forall t\in \mathbb{T}. \label{coinceq1}  \end{equation}
\end{theorem}

\begin{proof}
Define the possibly nonlinear operator $$\phi : L^2_a(\mu_{\mathbb{T}}, H_{\mathbb{T}})\to L^2_a(\mu_{\mathbb{T}},H_{\mathbb{T}})$$ by  $$\phi := I_{L^2_a(\mu_{\mathbb{T}}, H_{\mathbb{T}})} +\kappa_X,$$ and 
first notice that,  if $h\in L^2_a( \mu_{\mathbb{T}}, H_{\mathbb{T}})$ and $\epsilon$ is any real number associated to $h$ by $(ii)$, then $h^\epsilon := \epsilon \phi(h) + (1-\epsilon) h$ satisfies $$h^\epsilon - h = \epsilon \kappa_X[h],$$ so that from hypothesis $(ii)$ we obtain $$\|\phi(h^\epsilon) - \phi(h)\|_{L^2_a(\mu_{\mathbb{T}},H_{\mathbb{T}})}\leq \sqrt{K} \|h^\epsilon -h\|_{L^2_a(\mu_{\mathbb{T}},H_{\mathbb{T}})},$$ where $\sqrt{K}\in [0,1[$.  
Moreover, since $\kappa_X$ is continuous from $(i)$, $\phi$ is continuous as well. Therefore, the Clarke fixed point theorem (\cite{CLARKE1}) ensures that $\phi$ admits a fixed point $u\in 
L^2_a(\mu_{\mathbb{T}},H_{\mathbb{T}})$, such that $\phi(u)=u$. From the definition of $\phi$,  $u$ satisfies $$u = -\theta \circ \tau_{u}, \ \mu_{\mathbb{T}}-a.s.$$ which turns out to be  equivalent 
to \begin{equation} X\circ \tau_{u} =I_{C_{\mathbb{T}}}, \ \mu_{\mathbb{T}}-a.s.; \label{tcU1a}  \end{equation} in the previous equation, Lemma~\ref{Lemmabs} still ensures that the pullback 
is well defined. Define $$\nu:= {\tau_u}_\star \mu_{\mathbb{T}},$$ which satisfies $\nu \sim \mu_{\mathbb{T}}$.  Indeed, since  $\theta\in L^\infty_a(\mu_{\mathbb{T}},H_{\mathbb{T}})$ and ${\tau_{u}}_{\star}\mu_{\mathbb{T}}<<\mu_{\mathbb{T}}$, which we know from Lemma~\ref{Lemmabs} applied to $u\in L^2_a(\mu_{\mathbb{T}},H_{\mathbb{T}})$, we first obtain $\theta\circ \tau_{u}\in L^{\infty}_a(\mu_{\mathbb{T}}, H_{\mathbb{T}})$, so that $u\in L^\infty_a(\mu_{\mathbb{T}},H_{\mathbb{T}})$. Therefore  $\nu:={\tau_{u}}_{\star}\mu_{\mathbb{T}}\sim\mu_{\Delta}$ follows again from  Lemma~\ref{Lemmabs} now applied to $u\in L^\infty_a(\mu_{\mathbb{T}},H_{\mathbb{T}})$. Then, it is routine to check 
that \begin{equation} \label{uct1}  \tau_{u} \circ X = I_{C_{\mathbb{T}}}, \ \nu-a.s.,\end{equation}and therefore $\mu_{\mathbb{T}}-a.s.$, taking into account that $\nu \sim \mu_{\mathbb{T}}$.
Furthermore, given $t\in \mathbb{T}$, since $u\in L^{\infty}_a(\mu_{\mathbb{T}}, H_{\mathbb{T}})$, whenever $s\in [0,t]_{\mathbb{T}}$,  we have that $W_s + u_s$ is $\mathcal{F}_s^{\mu_{\mathbb{T}}}-$ measurable and therefore $\mathcal{F}_t^{\mu_{\mathbb{T}}}-$ measurable, which yields $$\tau_{u}^{-1}(A) \in \mathcal{F}_t^{\mu_{\mathbb{T}}}, \ \forall A\in \mathcal{F}_t,$$ where $\mathcal{F}_t=\sigma(W_s :s\in [0,t]_{\mathbb{T}})$.  Hence, whenever $A\in \mathcal{F}_t$,~(\ref{uct1}) yields 
$A\in \mathcal{G}_t^X$, so that we obtain  $$\mathcal{F}_t^{\nu}\subset \mathcal{G}_t^{X}, \ \forall t\in \mathbb{T},$$  while at the inverse the hypothesis $\theta\in L^{\infty}_a(\mu_{\mathbb{T}},H_{\mathbb{T}})$ yields $$\mathcal{G}_t^X\subset \mathcal{F}_t^{\nu},$$ from which~(\ref{coinceq1}) follows. On the other hand, from the definition of $\nu$,~(\ref{tcU1a}) yields $X_\star \nu = \mu_{\mathbb{T}}$, and therefore $(X_t)_{t\in \mathbb{T}}$ is a sampled Brownian motion which generates the canonical filtration. Notice that from the notation, $\mathcal{G}_t^X= \sigma(X_s,  s\in [0,t]_{\mathbb{T}})^\nu$, $\forall t\in \mathbb{T}$, and that with a symmetric proof as~(\ref{coinceq1}), $\mathcal{F}_t^{\mu_{\mathbb{T}}}= \sigma(W_s + u_s : s\in [0,t]_{\mathbb{T}})^{\mu_{\mathbb{T}}}$ also holds similarly, $\forall t\in \mathbb{T}$.
\end{proof}

\section{Applications to stochastic differential equations driven by a sampled Brownian motion.} \label{Fortetsection}

    In the pioneering work \cite{Fortet2} notably on absolutely continuous transformations of laws of  stochastic processes, within his investigations in the first half of the XXth century, and before the celebrated paper of K. Itô (\cite{ito44}), the french mathematician Robert Fortet proposed a fixed point method to handle a specific equation with a noise term, where appears a \textit{symbolic reciprocal formula} (see \cite{Fortet2} p.199).  Somehow astonishingly, this method which is expected to be much natural  within an analytic point of view is quite distinct from the nowadays classical It\^{o}'s fixed point method for stochastic differential equations (\cite{ITO3}). However, it turns out that Fortet's approach is much suitable for our purposes. Although to be made fully rigorous within the much refined framework of It\^{o}'s stochastic differential equations  for a large class of drift terms including non-markovians, Robert Fortet's intuition requires rather elaborated tools of transport of measure as isomorphisms mod 0  of Rokhlin  (\cite{Rohlin}), or even much specific pullbacks of morphisms of probability spaces, which require an accurate control of the absolute continuity of the deterministic transport plans involved, the intuition behind the Fortet symbolic formula
can be readily observed even within the elaborated time scales framework of this paper. Consider the stochastic differential equation driven by a sampled Brownian motion $(B_t)_{t\in \mathbb{T}}$,   \begin{equation} X_t = B_t +  \int_{[0,t)_{\mathbb{T}}} \beta_s(X) \lambda_{\Delta}(ds), \ \forall t\in \mathbb{T},  \label{SDEdef1} \end{equation}  where  $$\beta : (t,\omega) \in \mathbb{T}\times C_{\mathbb{T}} \to \beta_t(\omega) \in \mathbb{R},$$ denotes a $\mathcal{B}(\mathbb{T})\otimes \mathcal{B}(C_{\mathbb{T}})/ \mathcal{B}(\mathbb{R})-$ measurable function such that $\beta_t : \omega\in C_{\mathbb{T}} \to \mathbb{R}$ is $\sigma(W_s : s\in [0,t]_{\mathbb{T}})-$ measurable, $\forall t\in \mathbb{T}$. The precise definition we use within this specific context will be clear from the statement of the next corollary. For pairs  $(X,B)$ which satisfy~(\ref{SDEdef1}),  defining $$T: \omega\in C_{\mathbb{T}} \to \omega - \int_{[0,.)_{\mathbb{T}}}\beta_s(\omega) \lambda_{\Delta}(ds)   \in C_{\mathbb{T}},$$ equation~(\ref{SDEdef1}) now reads informally $$T(X)=B,$$ where $B$ and $X$ are now interpreted as random functions. Hence, if $F$ denotes a suitable reciprocal of $T$, and $(B_t)_{t\in \mathbb{T}}$ an $(\mathcal{A}_t)_{t\in \mathbb{T}}-$ sampled Brownian motion on a complete probability space $(\Omega, \mathcal{A},\mathcal{P})$,   one may hope that, under conditions, a formula of the form $X=F(B)$ should produce solutions $(X_{t})_{t\in\mathbb{T}}$ to~(\ref{SDEdef1}) from $F$ and $(B_t)_{t\in\mathbb{T}}$. For the sake of clarity, henceforth we focus on the case where $\beta$ is both bounded and adapted, as stated accurately below.  Within a practical point of view, a key point to apply rigorously within our stochastic framework  this method, which is familiar in analytic contexts,  is to obtain at some point a result equivalent to the coincidence of $\sigma-$ fields \begin{equation} \mathcal{G}_t^T =\mathcal{F}_t^\nu, \forall t\in \mathbb{T}, \label{emofi} \end{equation} where $\mathcal{G}_t^T:= (T^{-1}(\mathcal{F}_t))^\nu$, $\forall t\in \mathbb{T}$, for some $\nu\in M_1(C_{\mathbb{T}})$. One way to address the problem by methods of functional nonlinear analysis is to apply a fixed point theorem pointwise, at each $\omega\in C_{\mathbb{T}}$ possibly outside a $\mu_{\mathbb{T}}-$ negligible set (for instance see \cite{FUZ} when $\mathbb{T}=[0,1]$). As it may be seen from \cite{USTUZAKAIINV1}, thanks to the so-called Souslin-Lusin theorem and by using analytic sets (\cite{DM}), under suitable hypothesis, at first such approaches essentially  yield results equivalent to the condition $\mathcal{G}_1^T= \mathcal{F}_1^{\nu}$ from a fixed point theorem, for some $\nu\in M_1(C_{\mathbb{T}})$; therefore, it can be used to obtain the existence of extremal solutions to the corresponding Yershov problem which was recalled in the introduction. However, to obtain rigorously the full~(\ref{emofi}), using the pointwise approach, one may either localize the condition, or obtain results equivalent to the uniqueness of solutions to a corresponding Yershov problem among some specific probability measures absolutely continuous with respect to $\mu_{\mathbb{T}}$; see also the references recalled in the introduction.  By contrast here, we follow a global approach, where  the fixed point method is applied directly in $L^2_a(\mu_{\mathbb{T}}, H)$, and from this, we obtain~(\ref{emofi}) without any further conditioning nor localization. Those aspects are reflected in our hypothesis, which provide sufficient conditions under integral forms, which furthermore do not require to know the laws of solutions \textit{a priori}, in contrast notably with several pathwise hypothesis (for instance see \cite{FUZ}, \cite{USTUZAKAIINV1}).

\begin{corollary} \label{sdecorollary}
Let $\beta :  \mathbb{T} \times C_{\mathbb{T}} \to \mathbb{R}$ be a bounded Borel measurable function such that $\forall t\in \mathbb{T}$, $\beta_t : \omega \in C_{\mathbb{T}}\to \beta_t(\omega)\in \mathbb{R}$ is $\mathcal{F}_t-$ measurable, where $\mathcal{F}_t:=\sigma(W_s, s \in [0,t]_{\mathbb{T}})$, $\forall t\in \mathbb{T}$.  Further assume that 

\begin{enumerate}[(i)] 
\item the function  $$F: h\in L^2_a(\mu_{\mathbb{T}}, H_{\mathbb{T}}) \to \int_{[0,.)_{\mathbb{T}}} \beta_s\circ \tau_h \lambda_{\Delta}(ds) \in L^2_a(\mu_{\mathbb{T}}, H_{\mathbb{T}}),$$ is continuous.
\item  $\exists K\in [0,1[ $ such that $\forall h\in \mathcal{L}^2_a(\mu_{\mathbb{T}},H_{\mathbb{T}})$, $\exists \epsilon \in ]0,1]$ which satisfy

 $$I(h) \leq K\epsilon ^2 \int_{C_{\mathbb{T}}} \left(\int_{\mathbb{T}} \left( h^\Delta_s(\omega) - \beta_s(\omega +h(\omega)) \right)^2 \lambda_{\Delta}(ds)\right) \mu_{\mathbb{T}}(d\omega),$$ 
where $I(h)$ denotes the integral
$$I(h):=\int_{C_{\mathbb{T}}} \left( \int_{\mathbb{T}} \left(\beta_s(\omega + h^\epsilon(\omega)) -\beta_s(\omega +h(\omega))\right)^2 \lambda_{\Delta}(ds) \right)\mu_{\mathbb{T}}{(d\omega)},$$
and where $$h^\epsilon_s(\omega) :=(1-\epsilon) h_s(\omega) +\epsilon \int_{[0,s)_{\mathbb{T}}} \beta_u(\omega+h(\omega)) \lambda_{\Delta}(du), \ \forall s\in \mathbb{T}, \ \mu_{\mathbb{T}}-a.s..$$
\end{enumerate}
Then, given an $(\mathcal{A}_t)_{t\in \mathbb{T}}-$ sampled Brownian motion $(B_t)_{t\in \mathbb{T}}$ which is defined on a complete filtered probability space $(\Omega, 
\mathcal{A}, (\mathcal{A}_t)_{t\in \mathbb{T}},\mathcal{P})$, there exists an $(\mathcal{A}_t)-$ adapted stochastic process $(X_t)_{t\in \mathbb{T}}$ defined on $(\Omega, \mathcal
{A},\mathcal{P})$ such that  \begin{equation} X_t = B_t + \int_{[0,t)_{\mathbb{T}}} \beta_s(X) \lambda_{\Delta}(ds), \ \forall t\in \mathbb{T}, \ \mathcal{P}-p.s.. \label{SDE1} \end
{equation} Furthermore, $(X_t)_{t\in \mathbb{T}}$ is a strong solution of~(\ref{SDE1}) for $(B_t)_{t\in \mathbb{T}}$, that is, there exists a Borel measurable function $F : C_{\mathbb
{T}} \to C_{\mathbb{T}}$, such that $F^{-1}(\mathcal{F}_t)\subset \mathcal{F}_t^{\mu_{\mathbb{T}}}$, $\forall t\in \mathbb{T}$ and $$X_t= F_t(B), \ \forall t\in \mathbb{T}, \ \mathcal{P}-
p.s..$$  In particular $(B_t)_{t\in \mathbb{T}}$ and $(X_t)_{t\in \mathbb{T}}$ also satisfy the following  
coincidence of filtrations  \begin{equation} \label{filtrequalokart} (\mathcal{G}_t^X)_{t\in \mathbb{T}}=(\mathcal{G}_t^B)_{t\in \mathbb{T}},\end{equation} where $\mathcal{G}_t^{X}:= \sigma(X_s, s\in [0,t]_{\mathbb{T}})^{\mathbb{P}}$ and $\mathcal{G}_t^{B}:= \sigma(B_s, s\in [0,t]_{\mathbb{T}})^{\mathbb{P}}$, $\forall t\in \mathbb{T}$.  Moreover,   the law $p_X:=X_{\star}\mathcal{P}$ of $(X_t)_{t\in \mathbb{T}}$ on $C_{\mathbb{T}}$, as a random $C_{\mathbb{T}}$ valued function $X:\Omega \to C_{\mathbb{T}}$, is equivalent to the 
sampled Wiener measure $\mu_{\mathbb{T}}= B_{\star}\mathcal{P}$, which is the law of the sampled Brownian motion $(B_t)_{t\in\mathbb{T}}$, as a $C_{\mathbb{T}}-$ valued random function $B : \Omega \to C_{\mathbb{T}}$.
\end{corollary}
\begin{proof}
Define $\theta$ to be the element of $L^2_a(\mu_{\mathbb{T}}, H_{\mathbb{T}})$ such that $$\theta := -\int_{[0,.)_{\mathbb{T}}} \beta_s \lambda_{\Delta}(ds), \ \mu_{\mathbb{T}}-a.s..$$
Then, assumption $(i)$, respectively $(ii)$, imply that assumptions $(i)$, respectively $(ii)$, of Theorem~\ref{theorem2} are satisfied by $I_{C_{\mathbb{T}}}+ \theta$. Therefore, together with its proof, the latter ensures that  denoting $T:= I_{C_{\mathbb{T}}} + \theta$,  there exists a $F := I_{C_{\mathbb{T}}}+ u$, for some $u\in L^\infty_a(\mu_{\mathbb{T}},H_{\mathbb{T}})$, which is such that $$T\circ F = F\circ T = I_{C_{\mathbb{T}}}, \ \mu_{\mathbb{T}}-a.s.,$$ where we used that $\nu:=F_{\star}\mu_{\mathbb{T}}$ is equivalent to $\mu_{\mathbb{T}}$ from Lemma~\ref{Lemmabs}. Now, let $(B_t)_{t\in \mathbb{T}}$ be an $(\mathcal{A}_t)_{t\in\mathbb{T}}-$ sampled Brownian motion which is defined on a complete probability space $(\Omega, \mathcal{A}, \mathcal{P})$, with a complete filtration $(\mathcal{A}_t)_{t\in \mathbb{T}}$. We pick a measurable map $B : \Omega \to C_{\mathbb{T}}$, such that $B_t= W_t\circ B$, $\forall t\in \mathbb{T}$, $\mathcal{P}-a.s.$ and define $$X_t:= F_t(B), \forall t\in \mathbb{T}, \mathcal{P}-a.s.,$$  and $X:= F\circ B$, $\mathcal{P}-a.s.$. Then, we obtain $$T\circ X= T\circ (F\circ B) = (T\circ F)\circ B = B, \ \mathcal{P}-a.s.,$$ where the pullbacks are well defined since $B_{\star}\mathcal{P} =\mu_{\mathbb{T}}$ and $X_{\star}\mathcal{P}=\nu$, which show that $(X,B)$ satisfy~(\ref{SDE1}). Moreover, notice that since $(B_t)_{t\in\mathbb{T}}$ is $(\mathcal{A}_t)_{t\in\mathbb{T}}-$ adapted, and since from the proof of Theorem~\ref{theorem2}, $F_t$ is $\mathcal{F}_t^{\mu_{\mathbb{T}}}-$ measurable $\forall t\in \mathbb{T}$, we obtain that $X_t$ is $\mathcal{A}_t$ measurable, $\forall t\in \mathbb{T}$; the coïncidence of $\sigma-$ fields also follows similarly from the conditions $(\mathcal{G}_t^{T})= (\mathcal{F}_t^\nu)$, and $(\mathcal{G}_t^F)= (\mathcal{F}_t^{\mu_{\mathbb{T}}})$ of Theorem~\ref{theorem2} together with its proof.  \end{proof}

\begin{corollary}\label{Corollaire2}
For any time scale $\mathbb{T}$ such that $\{0,1\}\subset \mathbb{T}\subset [0,1]$, and $a\in ]-1,1[$, the stochastic differential equation driven by an $(\mathcal{A}_t)_{t\in\mathbb{T}}-$ sampled Brownian motion $(B_t)_{t\in\mathbb{T}}$ defined on a complete filtered probability space $(\Omega,\mathcal{A},$ $(\mathcal{A}_t)_{t\in\mathbb{T}}$, $\mathcal{P})$   \begin{equation} \label{exo1M} X_t = B_t + \int_{[0,t[_{\mathbb{T}}} a \sin(X_{s}-X_{\rho(s)}) \lambda_\Delta(ds), \ \forall t\in \mathbb{T}, \ \mathcal{P}-p.s.,\end{equation}
has a strong solution $(X_t)_{t\in\mathbb{T}}$, whose law $p_X\in M_1(C_{\mathbb{T}})$ is equivalent to the sampled Wiener measure $\mu_{\mathbb{T}}$ ;  $\rho :=\bar{\rho}|_{\mathbb{T}}$ denotes the time scales backward jump operator of $\mathbb{T}$.
\end{corollary}
\begin{proof}
Whenever $a\in]-1,1[$, define  $$\beta_t(\omega):= a \sin(\omega(t)- \omega(\rho(t))), \ \forall (t,\omega)\in \mathbb{T}\times C_{\mathbb{T}}.$$ Then, setting $K:=a^2$, and taking into account that $\lambda_{\Delta}(\mathbb{T})=1$, which allows to apply the Jensen inequality, it is an easy exercise to check that $\beta$ satisfies the hypothesis of Corollary~\ref{sdecorollary}. 
\end{proof}

We now provide examples with typical time scales.

\begin{example}
Within the convention on the notation of $H_{\mathbb{T}}$ adopted in this paper, the equation~(\ref{SDE1}) in Corollary~\ref{sdecorollary} boils down to  
\begin{itemize}
\item (Infinite countable time-scale)  $$X_{t_n}= B_{t_n} + \sum_{k<n } \beta_{t_k}(X)(t_{k+1}-t_k), \ \forall n\in \{0\} \cup \mathbb{Z}^{-},$$ if  $\mathbb{T}:=\{0\}\bigcup\{t_n : n \in\{0\}\cup\mathbb{Z}^{-} \}$, where $t_0=1$, and  where $(t_n)_{n\in \mathbb{Z}^-}$ denotes a  strictly decreasing sequence such that $\lim_{n\to -\infty} t_n =0$; notice that $0$ is right-dense which entails $\lambda_{\Delta}(\{0\})=0$. 
\item (Continuous time case)   $$X_t= B_t +\int_0^t \beta_s(X)\lambda(ds), \ \forall t\in [0,1],$$ if $\mathbb{T}= [0,1]$, $\lambda$ denoting the classical Lebesgue measure.
\end{itemize}
Moreover, in both cases, within models the drift term may depend on the whole past of $X$.
\end{example}

\begin{example}(Cantor sampling of brownian filtrations.)
Denote by $$K_3:= \left\{\sum_{n=1}^{+\infty}\frac{z_n}{3^n} : z_n\in\{0,2\}, \ \forall n\geq 1\right\},$$ the Cantor set, then $\mathbb{T}:=K_3$ satisfies the conditions of Corollary~\ref{Corollaire2}. Let $(B_t)_{t\in[0,1]}$ be a standard Brownian motion on a complete probability space $(\Omega, \mathcal{A},\mathcal{P})$,  then there exists a stochastic process $(X_t)_{t\in K_3}$ which satisfies $$X_t = B_t + \frac{1}{2} \int_{[0,t[_{\mathbb{T}}}  \sin(X_{s}-X_{\rho(s)}) \lambda_\Delta(ds), \ \forall t\in K_3, \ \mathcal{P}-p.s..$$ Moreover, the stochastic process  $(X_t)_{t\in K_3}$, which is defined on $(\Omega, \mathcal{A},\mathcal{P})$, generates the Cantor sampled brownian filtration $(\mathcal{G}_t^B)_{t\in K_3}$, that is~(\ref{filtrequalokart}) holds.
\end{example}

\par\bigskip\noindent

\bibliographystyle{amsplain}

\noindent

\end{document}